\documentclass[a4paper,11pt]{article}

\usepackage{amsmath,amsthm,amscd,amssymb,latexsym,graphicx}

\usepackage{geometry, fullpage}
\usepackage[parfill]{parskip} 

\numberwithin{equation}{section}

\theoremstyle{plain}
\newtheorem{theorem}{Theorem}[section]
\newtheorem{lemma}[theorem]{Lemma}
\newtheorem{corollary}[theorem]{Corollary}
\newtheorem{proposition}[theorem]{Proposition}

\theoremstyle{definition}
\newtheorem{definition}[theorem]{Definition}

\newtheorem{example}[theorem]{Example}

\theoremstyle{remark}

\newtheorem{case[theorem]}{Case}

\def\R{{\mathbb R}}
\def\Z{{\mathbb Z}}

\def\T{{\mathbb T}}
\def\G{{\mathcal G}}
\def\HD{{\mathbb H^d}}
\def\F{{\mathcal F}}
\def\P{{\mathbb P}}
\def\H{{\mathcal H}}
\def\HS{{\mathcal{HS}}}
\def\z{{\mathfrak z}}
\def\g{{\bf g}}
\def \l {\lambda}

\def\A{{\mathbb A}}
\def\B{{\mathbb B}}

\def\norm#1.#2.{\lVert#1\rVert_{#2}}

\def\R{\mathbb R}

\newcommand{\psis}[2]{\langle #1\rangle_{#2}}
\newcommand{\uir}[1]{\pi_{#1}}
\newcommand{\lr}{L}

\begin{document}

\title{Bracket map for Heisenberg group and the characterization of cyclic subspaces}
\date{\today}

\author{Davide Barbieri, Eugenio Hernandez, Azita Mayeli}

\maketitle

\begin{abstract}
The bracket map was originally considered in \cite{HSWW} for locally compact abelian groups. In this work we extend
the study of bracket maps to the noncommutative setting, providing characterizations of bases and frames for cyclic subspaces of the Heisenberg group. We also indicate how to generalize these results to  a  class of non-abelian nilpotent Lie groups whose irreducible representations are square integrable modulo the center.
\footnote{
\thanks{The research of the first author was supported by Grant DIM2011 - R\'egion \^{I}le de France. The research of the second author was supported by Grant MTM2010-16518 (Spain). The research of the third author was partially supported by PSC-CUNY Grant 65508-0043 (USA).}


}

\end{abstract}

\section{Introduction}

Principal shift invariant subspaces obtained as the closed linear span of integer translates of a single function $\psi \in L^2(\R^d)$ have been studied recently in connection with the theory of Multiresolution Analysis for wavelets.

The general situation could be described as follows. Let $\G$ be a topological group, and denote by $g \mapsto T_g$ a unitary representation of $\G$ on a Hilbert space $\H$. The Principal Shift Invariant (PSI) subspace generated by $\psi \in \H$, denoted by $\psis{\psi}{\G,T}$, is the closure in $\H$ of the linear span generated by $\{T_g \psi\}_{g \in \G}$, and one could ask to characterize those $\psi \in \H$ so that the collection $\{T_g \psi\}_{g \in \G}$ is an orthonormal basis, a Riesz basis or a frame for $\psis{\psi}{\G,T}$.

Such characterizations have been provided for countable abelian groups $\G$ equipped with the discrete topology, whose unitary representations are dual integrable (see \cite{HSWW} for details). Such characterizations are given in terms of a bracket map
$$
[\cdot , \cdot] : \H \times \H \rightarrow L^1(\hat{\G}, d\alpha)
$$
where $\hat{\G}$ is the character group of $\G$ with Haar measure $d\alpha$, such that
\begin{equation}\label{equ:dual-integrable}
\langle \varphi, T_g \psi \rangle = \int_{\hat{\G}} [\varphi, \psi](\alpha) \overline{\alpha(g)} d\alpha \quad \forall \ \varphi, \psi \in \H .
\end{equation}
Two classical examples of bracket maps are the following. The first one is the map $k \mapsto T_k$ given by $T_k f(x) = f(x + k)$ as a unitary representation of $(\Z^d,+)$ on $L^2(\R^d)$. In this case
\begin{equation}\label{equ:commutative-translations-bracket}
[\varphi, \psi](\xi) = \sum_{l \in \Z^d} \hat{\varphi}(\xi + l) \overline{\hat{\psi}(\xi + l)}, \quad \xi \in \R^d, \ \varphi, \psi \in L^2(\R^d)
\end{equation}
where $\hat{\psi}$ stands for the Fourier transform of $\psi$. The bracket map for this situation has appeared in \cite{BVR1, BVR2}.
The second one is the Gabor unitary representation of the group $(\Z^d \times \Z^d,+)$ on $L^2(\R^d)$ given by
$$
(k,l) \mapsto M_l T_k \psi(x) = e^{-2\pi i l\cdot x} \psi(x - k).
$$
In this case
\begin{equation}\label{equ:commutative-gabor-bracket}
[\varphi, \psi](x,\xi) = Z\varphi(x,\xi)\overline{Z\psi(x,\xi)}, \quad x,\xi \in \R^d, \ \varphi, \psi \in L^2(\R^d)
\end{equation}
where
$$
Z\psi(x,\xi) = \sum_{k \in \Z^d} e^{-2\pi i k \cdot \xi} \psi(x - k)
$$
is the Zak transform of the function $\psi \in L^2(\R^d)$. The bracket map for this situation has appeared in \cite{HW1989, HP2006}.

An appearence of the bracket map in disguise is in \cite{AAR} where invariant subspaces for local fields are studied with respect to suitable defined translations.

As far as we know, no attempt has been done to provide a bracket map for non-Abelian groups. Our purpose in this paper is to provide such a map for left translations on the Heisenberg group. A characterization of frames and Riesz bases for left translations on nilpotent Lie groups whose irreducible representations are square integrable modulo the center (which include the Heisenberg group), also called $SI/Z$ Lie groups, has been recently announced in \cite{CMO}. These characterizations are in terms of range functions, rather than bracket maps.

In this paper the Heisenberg group $\HD$ is identified with $\R^d \times \R^d \times \R$ and noncommutative group law given by
\begin{equation}\label{equ:Hlaw}
(p,q,t) (p',q',t') = (p + p', q + q', t + t' + p\cdot q').
\end{equation}
where $\cdot$ stands for $\R^d$ scalar product. It is well known that inequivalent unitary irreducible representations $\Pi$ of $\HD$ on a Hilbert space $\H$ are indexed by $\lambda \in \R \setminus \{0\}$ (see \cite{F1995}). With them we can define the Fourier transform of $\varphi \in L^2(\HD)$ as
\begin{equation}\label{equ:Hfourier}
\F_\HD\varphi(\lambda) = \int_\HD \varphi(x) \Pi_\lambda(x) dx
\end{equation}
and those are Hilbert-Schmidt operators on $\H$.

Our definition of the bracket map for the Heisenberg group is the following.

\begin{definition}\label{def:HDbracket}
Given $\varphi, \psi \in L^2(\HD)$ and $\alpha \in (0,1]$, we define
\begin{equation}\label{equ:Hbracket}
[\varphi, \psi](\alpha) = \sum_{j \in \Z} \langle \F_\HD \varphi(\alpha + j), \F_\HD \psi(\alpha + j)\rangle_{\HS} |\alpha + j|^d
\end{equation}
where $\langle \cdot , \cdot \rangle_{\HS}$ denotes the inner product on the space of Hilbert-Schmidt operators on $\H$.
\end{definition}

As we will see, the appearence of $|\alpha + j|^d$ is related to the Plancherel measure on the unitary dual of $\HD$. Such bracket map will allow us to study PIS subspaces with respect to left translations of discrete subgroups of $\HD$.

Left translations $\lr$ are unitary representations on $L^2(\HD)$ defined by
\begin{equation}\label{equ:leftregular}
g \mapsto \lr_g\psi(x) = \psi(g^{-1} x), \quad g, x \in \HD, \ \psi \in L^2(\HD)
\end{equation}
and we will consider lattice subgroups $\Gamma$ of $\HD$ with the following form: for $\A, \B \in GL(d,\R)$ matrices such that $\A\B^t \in \Z$, we let
\begin{equation}\label{eq:lattice}
\Gamma= \A\Z^d\times \B\Z^d\times \Z .
\end{equation}
In particular, we will distinguish the center of $\Gamma$ as a group from the other variables, denoting $\Gamma_1= \A\Z^d \times \B\Z^d$ and $\Gamma_2 = \Z$. This permits to write $\Gamma = \Gamma_1\Gamma_2$ in the sense that any element $\gamma \in \Gamma$ can be written as $\gamma = \gamma_1 \gamma_2$, where $\gamma_i \in \Gamma_i$, $i = 1,2$ and composition is made according to (\ref{equ:Hlaw}).

Properties of the bracket map (\ref{equ:Hbracket}) are given in Section \ref{sec:HDbracket}, where the following result is proved.

\begin{theorem}\label{theo:orthogonal} Let $\psi$ be in $L^2(\HD)$, $\psi \neq 0$. Then the family $\{\lr_\gamma\psi\}_{\gamma \in \Gamma}$ is an orthonormal basis for $\psis{\psi}{\Gamma,\lr}$ if and only if
\begin{equation}\label{equ:orthonormalitycondition}
[\psi, \lr_{\gamma_1}\psi](\alpha) = \delta_{\gamma_1,0} \quad \textnormal{a.e.} \ \alpha \in (0,1], \ \forall \ \gamma_1\in \Gamma_1 .
\end{equation}
\end{theorem}

Section \ref{sec:technical results} is devoted to prove a technical result of crucial importance for the results in Section \ref{sec:Riesz bases and frames}. Namely, for $\varphi, \psi \in L^2(\HD)$ we define
\begin{equation}\label{equ:Sfunctional}
S_\psi[\varphi](\alpha) \doteq \chi_{E_\psi}(\alpha) \frac{[\varphi,\psi](\alpha)}{[\psi,\psi](\alpha)} , \quad \alpha \in (0,1]
\end{equation}
where
$$
E_\psi = \{\alpha \in (0,1] : [\psi,\psi](\alpha) > 0\}
$$
and prove the following:

\begin{theorem}\label{theo:Isometry}
The map $S_\psi$ is a linear isometry from
$$
\psis{\psi}{\Gamma_2, L} = \overline{\textnormal{span}\{\lr_{\gamma_2}\psi : \gamma_2\in \Gamma_2\}}^{L^2(\HD)}
$$
onto the weighted Hilbert space $L^2((0,1],[\psi,\psi](\alpha)d\alpha)$.
\end{theorem}

Section \ref{sec:Riesz bases and frames} is devoted to find conditions on $\psi$ such that its left translations form a Riesz basis or a frame of the space $\psis{\psi}{\Gamma,\lr}$. Such conditions in terms of the bracket map (\ref{equ:Hbracket}) are given by the following theorems.

\begin{theorem}\label{theo:Riesz}
Let $\psi$ be in $L^2(\HD)$, $\psi \neq 0$ and assume assume that
\begin{equation}\label{equ:Condition}
[\psi, \lr_{\gamma_1}\psi] (\alpha) = 0 \quad \textnormal{a.e.} \ \alpha \in E_\psi, \ \forall \ \gamma_1\in \Gamma_1 \setminus \{0\} .
\end{equation}
Then for $0 < A \leq B < \infty$ the following are equivalent
\begin{itemize}
\item[a)] $\{\lr_\gamma\psi\}_{\gamma \in \Gamma}$ is a Riesz basis for $\psis{\psi}{\Gamma,\lr}$ with constants $A$ and $B$
\item[b)] $A \leq [\psi,\psi](\alpha) \leq B$ for a.e. $\alpha \in (0,1]$.
\end{itemize}
\end{theorem}

\newpage

\begin{theorem}\label{theo:frames}
Let $\psi$ be in $L^2(\HD)$, $\psi \neq 0$ and assume (\ref{equ:Condition}). Then for $0 < A \leq B < \infty$
the following are equivalent
\begin{itemize}
\item[a)] $\{\lr_\gamma\psi\}_{\gamma \in \Gamma}$ is a frame for $\psis{\psi}{\Gamma,\lr}$ with constants $A$ and $B$
\item[b)] $A \leq [\psi,\psi](\alpha) \leq B$ for a.e. $\alpha \in E_\psi$.
\end{itemize}
\end{theorem}

In Section \ref{sec:Applications} we will then provide a construction of frames for cyclic subspaces, based on Theorem \ref{theo:frames}, while Section \ref{sec:SIZ} is devoted to show how the presented results can be generalized to the more abstract setting of $SI/Z$ groups.

{\bf Acknowledgements}. This work started when the second and third
authors participated in the \textit{Concentration week on Frame Theory and Maps bewteen Operator Algebras} held on July
16-22, {{2012,}} at Texas A\&M University. We
 express our gratitude to the Organizing Committee for
the invitation to participate in this meeting. 
\medskip

\section{Properties of the bracket map on the Heisenberg group}\label{sec:HDbracket}

We denote by $\HD\equiv \R^d\times \R^d\times \R$ the $d$-dimensional Heisenberg group, with (non-commutative) composition law given by (\ref{equ:Hlaw}), also addressed as \emph{polarized} Heisenberg group \cite{F1995}. Note that with this law, the inverse of an element is $(p,q,t)^{-1}=(-p,-q,-t+p\cdot q)$.

The Haar measure of $\HD$ is the Lebesgue measure on $\R^{2d+1}$, and the center of $\HD$ is the subgroup $Z=\R$ of all elements $(0,0,t)$, so any $x= (p,q,t)$ can be written as $x = (p,q,0)(0,0,t)  = (0,0,t)(p,q,0)$.

The class of non-zero measure irreducible representations of the Heisenberg group needed to build up the Plancherel theorem are identified by $\R^*=\R\setminus\{0\}$. Then $\widehat{\HD}=\R^*$ up to a zero measure set. For any  $\lambda\in \R^*$ and any $x=(p,q,t)\in \HD$, such representations are equivalent to Schr\"odinger representations on the Hilbert space $L^2(\R^d)$, that we denote with the map $\uir{\lambda}(x)$
\begin{equation}\label{definition-of-schroedinger-representation}
\uir{\lambda}(x)f(y)= \uir{\lambda}(p,q,t)f(y)= e^{2\pi i \lambda t} e^{-2\pi i  \lambda q \cdot y} f(y-p)= e^{2\pi i \lambda t} \uir{\lambda}(p,q,0) f(y)
\end{equation}
on $f\in L^2(\R^d)$, where $\uir{\lambda}(p,q,0) = M_{\lambda q} T_{p} $ is the shift-frequency operator on $L^2(\R^d)$.

For $\varphi\in L^2(\HD)$ the Hilbert-Schmidt operator-valued Fourier transform $\F_\HD\varphi$ is given by
\begin{displaymath}
\F_\HD\varphi(\lambda) = \int_{\HD} \varphi(x) \uir{\lambda}(x) dx \ , \quad \lambda \in \R^* .
\end{displaymath}

It is well known (see e.g. \cite{F1995}) that the Plancherel measure on $\widehat\HD$ is given by $|\lambda|^d d\lambda$, so that Plancherel theorem reads
\begin{equation}\label{eq:Plancherel}
\langle \varphi, \psi\rangle_{L^2(\HD)} = \int_\R \langle \F_\HD \varphi(\lambda), \F_\HD \psi(\lambda)\rangle_{\HS} |\lambda|^d d\lambda
\end{equation}
where
\begin{displaymath}
\langle \F_\HD \varphi(\lambda), \F_\HD \psi(\lambda)\rangle_{\HS} = \textnormal{trace}_{L^2(\R^d)} \left(\F_\HD \varphi (\lambda) \left(\F_\HD \psi(\lambda)\right)^\dag\right)
\end{displaymath}
and $\left(\F_\HD \psi(\lambda)\right)^\dag$ stands for the $L^2(\R^d)$ adjoint of $\F_\HD \psi(\lambda)$.
\newpage
We will now prove some key properties of the bracket map of Definition \ref{def:HDbracket}.

\begin{lemma}\label{lem:sesquilinear}
The map (\ref{equ:Hbracket}) is a sesquilinear map from $L^2(\HD)\times L^2(\HD)$ to $L^1((0,1])$.
\end{lemma}
\begin{proof}
The sesquilinearity follows from the linearity of the Fourier transform and the sesquilinearity of the Hilbert-Schmidt scalar product. To show that $[\varphi,\psi] \in L^1((0,1])$ we show that
$$
\int_0^1 |[\varphi,\psi](\alpha)| d\alpha \leq \|\varphi\|_{L^2(\HD)} \|\psi\|_{L^2(\HD)}.
$$
Indeed
\begin{eqnarray*}
\int_0^1 |[\varphi,\psi](\alpha)| d\alpha  & = & \int_0^1 |\sum_{j \in \Z} \langle \F_\HD \varphi(\alpha + j), \F_\HD \psi(\alpha + j)\rangle_{\HS}|\, |\alpha + j|^d d\alpha \\
& \hspace{-120pt} \leq & \hspace{-65pt}\int_0^1 \sum_{j \in \Z} \|\F_\HD \varphi(\alpha + j)\|_{\HS} \, \|\F_\HD \psi(\alpha + j)\|_{\HS} \, |\alpha + j|^d d\alpha \\
& \hspace{-120pt} \leq & \hspace{-65pt} \left(\int_0^1 \sum_{j \in \Z} \|\F_\HD \varphi(\alpha + j)\|^2_{\HS}|\alpha + j|^d d\alpha\right)^\frac12 \left(\int_0^1 \sum_{j \in \Z} \|\F_\HD \psi(\alpha + j)\|^2_{\HS} \, |\alpha + j|^d d\alpha \right)^\frac12 \\
& \hspace{-120pt} = & \hspace{-65pt} \left(\int_\R \|\F_\HD \varphi(\lambda)\|^2_{\HS}|\lambda|^d d\lambda\right)^\frac12 \left(\int_\R \|\F_\HD \psi(\lambda)\|^2_{\HS} \, |\lambda|^d d\lambda \right)^\frac12
\end{eqnarray*}
and by Plancherel theorem we can conclude.
\end{proof}

\begin{lemma}\label{lem:CS}
For any $\varphi, \psi \in L^2(\HD)$ we have
$$
|[\varphi,\psi](\alpha)|^2 \leq [\varphi,\varphi](\alpha) [\psi,\psi](\alpha) \quad \textnormal{a.e.} \ \alpha \in (0,1] .
$$
\end{lemma}
\begin{proof}
By Cauchy-Schwartz inequality
\begin{eqnarray*}
|[\varphi,\psi](\alpha)|^2 & \leq & \left(\sum_{j \in \Z} \|\F_\HD \varphi(\alpha + j)\|^2_{\HS}|\alpha + j|^d \right) \left( \sum_{j \in \Z} \|\F_\HD \psi(\alpha + j)\|^2_{\HS} \, |\alpha + j|^d \right) \\
& = & [\varphi,\varphi](\alpha) [\psi,\psi](\alpha) .
\end{eqnarray*}
\end{proof}

The following lemma provides two properties of the bracket map that will be intensively used.
\begin{lemma}\label{lem:unitaryhom}
Let $\varphi, \psi$ be in $L^2(\HD)$. Then
\begin{itemize}
\item[i)] $[\lr_{x}\varphi, \lr_{x'}\psi](\alpha) = [\varphi, \lr_{x^{-1}x'}\psi](\alpha)$ for all $x,x' \in \HD$\\
\item[ii)] $[\varphi,\lr_{\gamma_2}\psi](\alpha) = e^{-2\pi i \gamma_2 \alpha} [\varphi,\psi](\alpha)$ for a.e. $\alpha \in (0,1]$ and all $\gamma_2 \in \Gamma_2$.
\end{itemize}
\end{lemma}

\begin{proof}
To prove $i)$, note first that for all $x \in \HD$ (see e.g. \cite{Fuhr})
\begin{equation}\label{equ:exchange}
\F_\HD \lr_x \psi (\lambda) = \uir{\lambda}(x) \F_\HD \psi (\lambda) .
\end{equation}
Since the operators $\uir{\lambda}(x)$, $x \in \HD$ are unitary homomorphisms, then using (\ref{equ:exchange})
\begin{eqnarray*}
\langle \F_\HD \lr_x \varphi(\lambda), \F_\HD \lr_{x'} \psi(\lambda)\rangle_{\HS} & = & \langle \uir{\lambda}(x) \F_\HD \varphi(\lambda), \F_\HD \lr_{x'} \psi(\lambda)\rangle_{\HS}\\
& = & \langle \F_\HD \varphi(\lambda), \uir{\lambda}(x^{-1}) \F_\HD \lr_{x'} \psi(\lambda)\rangle_{\HS}\\
& = & \langle \F_\HD \varphi(\lambda), \F_\HD \lr_{x^{-1}x'} \psi(\lambda)\rangle_{\HS}
\end{eqnarray*}
so the proof follows by the definition (\ref{equ:Hbracket}) of the bracket.

To prove $ii)$ we can use the same argument, noting that $\uir{\lambda}(\gamma_2) = e^{2\pi i \gamma_2 \lambda}$ .
\end{proof}

Next lemma provides the key relation with the notion of dual integrability introduced in \cite{HSWW}.
\begin{lemma}\label{lem:dualintegable}
The abelian group $\Gamma_2$ is dual integrable in the sense of \cite{HSWW} on $L^2(\HD)$ with respect to the bracket (\ref{equ:Hbracket}), i.e.
for all $\gamma_2 \in \Gamma_2 = \mathbb Z$ we have
$$
\langle \varphi, \lr_{\gamma_2}\psi\rangle_{L^2(\HD)} = \int_0^1 [\varphi, \psi] e^{-2\pi i \gamma_2 \alpha} d\alpha \quad \forall \ \varphi, \psi \in L^2(\HD).
$$
\end{lemma}
\begin{proof}
By Plancherel theorem and a periodization on $\R$
\begin{align*}
\langle & \varphi, \lr_{\gamma_2}\psi\rangle_{L^2(\HD)} = \int_\R \langle \F_\HD \varphi(\lambda), \F_\HD \lr_{\gamma_2} \psi(\lambda)\rangle_{\HS} |\lambda|^d d\lambda\\
& = \int_0^1 \sum_{j \in \Z} \langle \F_\HD \varphi(\alpha + j), \F_\HD \lr_{\gamma_2}\psi(\alpha + j)\rangle_{\HS} |\alpha + j|^d d\alpha
= \int_0^1 [\varphi, \lr_{\gamma_2}\psi] d\alpha\ .
\end{align*}
The proof then follows by Lemma \ref{lem:unitaryhom}, $ii)$, noting that $e^{2\pi i \gamma_2 j} = 1$.
\end{proof}

As a corollary of Lemma \ref{lem:dualintegable} we have the following relationship between the left regular representation of discrete subgroups (\ref{eq:lattice}) with the bracket map (\ref{equ:Hbracket}).
\begin{corollary}
For all $\varphi, \psi \in L^2(\HD)$ and all $\gamma = \gamma_1 \gamma_2  \in \Gamma$ we have
\begin{equation}\label{equ:lrr-bracket}
\langle \varphi, \lr_{\gamma}\psi\rangle_{L^2(\HD)} = \int_0^1 [\varphi, \lr_{\gamma_1}\psi] e^{-2\pi i \gamma_2 \alpha} d\alpha .
\end{equation}
\end{corollary}
\begin{proof}
The proof follows by applying Lemma \ref{lem:dualintegable} to $\lr_{\gamma_1}\psi$ instead of $\psi$.
\end{proof}

We can now provide an equivalent characterization of condition (\ref{equ:Condition}), that will be used in the proofs of Theorems \ref{theo:Riesz} and \ref{theo:frames}.

\begin{proposition}\label{prop:Condition}
Let $\psi \in L^2(\HD)$. Then condition (\ref{equ:Condition}) is equivalent to
\begin{equation}\label{equ:ConditionEquiv1}
\langle \psi, \lr_{\gamma_1\gamma_2}\psi\rangle_{L^2(\HD)} = \delta_{\gamma_1,0} \langle \psi, \lr_{\gamma_2}\psi\rangle_{L^2(\HD)} \quad \forall \gamma = \gamma_1\gamma_2 \in \Gamma .
\end{equation}
\end{proposition}
\begin{proof}
We first note that, by Lemma \ref{lem:CS}, whenever $\alpha \notin E_\psi$ then $[\psi,\lr_\gamma\psi] = 0$ for all $\gamma$. Hence the vanishing of the bracket in condition (\ref{equ:Condition}) can be equivalently stated for a.e. $\alpha \in (0,1]$.

We prove the equivalence of (\ref{equ:Condition}) with (\ref{equ:ConditionEquiv1}) when $\gamma_1 \neq 0$, since otherwise (\ref{equ:ConditionEquiv1}) is always true. Let us suppose that (\ref{equ:Condition}) holds. Then (\ref{equ:ConditionEquiv1}) follows by (\ref{equ:lrr-bracket}) and the previous observation. On the contrary, let us suppose that (\ref{equ:ConditionEquiv1}) holds. Then again by (\ref{equ:lrr-bracket}) and the uniqueness of Fourier series coefficients we have that $[\psi, \lr_{\gamma_1}\psi] = 0$ for a.e. $\alpha \in (0,1]$.
\end{proof}

We are ready to prove the theorem concerning orthonormal systems.

\begin{proof}[Proof of Theorem \ref{theo:orthogonal}]
Let us first suppose that $\{\lr_\gamma \psi\}_{\gamma \in \Gamma}$ is an orthonormal system. Then by (\ref{equ:lrr-bracket})
$$
\delta_{\gamma, 0} = \langle \psi, \lr_\gamma \psi\rangle_{L^2(\HD)} = \int_0^1 [\psi, \lr_{\gamma_1}\psi](\alpha) e^{-2\pi i \alpha \gamma_2} d\alpha .
$$
By Lemma \ref{lem:sesquilinear}, $[\psi, \lr_{\gamma_1}\psi]$ is in $L^1((0,1])$, so the result follows by uniqueness of Fourier series coefficients.

Conversely, let us suppose that $[\psi,\lr_{\gamma_1}\psi](\alpha) = \delta_{\gamma_1,0}$ for a.e. $\alpha \in (0,1]$ and all $\gamma_1 \in \Gamma_1$. Then if we call $\gamma'' = \gamma^{-1}\gamma'$, and decompose it as $\gamma'' = \gamma_1''\gamma_2''$ with $\gamma_i'' \in \Gamma_i, i=1,2$
\begin{eqnarray*}
\langle \lr_\gamma\psi, \lr_{\gamma'} \psi\rangle_{L^2(\HD)} & = & \langle \psi, \lr_{\gamma^{-1}\gamma'} \psi\rangle_{L^2(\HD)} = \int_0^1 [\psi, \lr_{\gamma_1''}\psi](\alpha) e^{-2\pi i \alpha \gamma_2''} d\alpha\\
& = & \delta_{\gamma_1'',0} \int_0^1 e^{-2\pi i \alpha \gamma_2''} d\alpha = \delta_{\gamma_1'',0}\delta_{\gamma_2'',0}
\end{eqnarray*}
so the result follows since $\delta_{\gamma_1'',0}\delta_{\gamma_2'',0} = \delta_{\gamma,\gamma'}$.
\end{proof}

\section{An isometric isomorphism}\label{sec:technical results}

Let $\psi$ be in $L^2(\HD)$, and call $E_\psi: = \{\alpha\in \T: \ [\psi, \psi](\alpha)>0\}$. Recall that the map $S_\psi$ at $\varphi\in L^2(\HD)$ was defined in (\ref{equ:Sfunctional}) as
$$
S_\psi[\varphi] (\alpha):= \chi_{E_\psi}(\alpha)\cfrac{[\varphi, \psi] (\alpha)}{[\psi ,  \psi](\alpha)} \quad \textrm{a.e.} \ \alpha\in (0,1].
$$

The following lemma shows how $S_\psi$ acts on $\Gamma$ shifts of a function.
\begin{lemma}\label{finite-sum}
Let $\psi$ be in $L^2(\HD)$. Then
\begin{itemize}
\item[i)] for all $\varphi \in L^2(\HD)$ and all $\Gamma \ni \gamma = \gamma_1\gamma_2$ with $\gamma_i \in \Gamma_i$, $i = 1,2$
\begin{displaymath}
S_\psi[\lr_{\gamma}\varphi](\alpha) = e^{2\pi i \gamma_2\alpha}S_\psi[\lr_{\gamma_1}\varphi](\alpha) \quad \textrm{a.e.} \ \alpha\in (0,1]
\end{displaymath}
\item[ii)] if $\varphi= \displaystyle{\sum_{\gamma_2\in \Gamma_2} a_{\gamma_2} \lr_{\gamma_2} \psi}$ be a finite sum with complex coefficients  $a_{\gamma_2}$, then
$$
S_\psi[\varphi](\alpha) = S_\psi\bigg[\sum_{\gamma_2\in \Gamma_2} a_{\gamma_2}  \lr_{\gamma_2}\psi\bigg](\alpha) = \bigg(\sum_{\gamma_2\in \Gamma_2}a_{\gamma_2} e^{2\pi i \gamma_2 \alpha}\bigg) \chi_{E_\psi}(\alpha) \quad \textrm{a.e.} \ \alpha\in (0,1].
$$
\end{itemize}
\end{lemma}
\begin{proof}
The proof of $i)$ follows from Lemma \ref{lem:unitaryhom}, $ii)$. To prove $ii)$, we see that $i)$
extends to finite sums by the linearity of the bracket map with respect to its first argument. Moreover, for $\gamma_2$ shifts we end up with $S_\psi[\varphi](\alpha)$, equal to $\chi_{E_\psi}(\alpha)$.
\end{proof}

\begin{proof}[Proof of Theorem \ref{theo:Isometry}]
The linearity of the map follows from its definition and the linearity of the  bracket function $[\cdot, \cdot]$ on its first argument. Let $\varphi\in L^2(\HD)$. Using the definition of $S_\psi$ and Lemma \ref{lem:CS} we have
\begin{eqnarray}
\int_0^1 |S_\psi[\varphi](\alpha)|^2 [\psi, \psi](\alpha) d\alpha& = & \int_{E_\psi} \frac{|[\varphi, \psi](\alpha)|^2}{ [\psi, \psi](\alpha)} \ d\alpha \leq \int_{E_\psi} \frac{[\varphi, \varphi](\alpha) [\psi, \psi](\alpha)}{ [\psi, \psi](\alpha)}\  d\alpha\label{equ:inequality}\\
& = & \int_0^1 [\varphi, \varphi](\alpha)  \  d\alpha = \|\varphi\|^2\nonumber
\end{eqnarray}
where the last identity is due to Lemma \ref{lem:dualintegable}.
We have thus proved that the image of the map $S_\psi$ is in the  Hilbert space $L^2(\T, [\psi,\psi](\alpha)d\alpha)$.

To prove that the restriction of $S_\psi$ to the space $\psis{\psi}{\Gamma_2,\lr}$ is an isometry (hence one-to-one) onto $L^2((0,1], [\psi,\psi](\alpha)d\alpha)$, we apply Lemma \ref{finite-sum} as follows.
If $\varphi\in \textnormal{span}\{\lr_{\gamma_2}\psi : \gamma_2\in \Gamma_2\}$, then we can write $\varphi = \displaystyle{\sum_{\gamma_2\in \Lambda}} a_{\gamma_2} \lr_{\gamma_2}\psi$ for some finite sequence $\{a_{\gamma_2}\}_{\gamma_2 \in \Lambda \subset \Gamma_2}$. Now by the sesquilinearity of $S_\psi$ and using Lemma \ref{lem:unitaryhom}, $ii)$
\begin{align*}
\int_0^1 \bigg| & S\bigg[\sum_{\gamma_2\in \Lambda} a_{\gamma_2}\lr_{\gamma_2}\psi\bigg](\alpha)\bigg|^2 \ [\psi,\psi](\alpha)d\alpha
\, = \int_0^1 \bigg|\sum_{\gamma_2} a_{\gamma_2} e^{2\pi i k \alpha}\bigg|^2\ [\psi,\psi](\alpha)d\alpha\\
& = \int_0^1 \left[\sum_{\gamma_2} a_{\gamma_2} \lr_{\gamma_2}\psi\ ,\  \sum_{\gamma_2} a_{\gamma_2} \lr_{\gamma_2}\psi\right](\alpha) \ d\alpha = \int_0^1 [\varphi,\varphi](\alpha) d\alpha = \|\varphi\|^2 .
\end{align*}
For general $\varphi \in \psis{\psi}{\Gamma_2, \lr}$ we have that for any $\epsilon > 0$ we can choose a finite $\Lambda_\epsilon \subset \Gamma_2$ and $\varphi_\epsilon = \displaystyle{\sum_{\gamma_2\in \Lambda_\epsilon}} a_{\gamma_2} \lr_{\gamma_2}\psi$ such that $\|\varphi - \varphi_\epsilon\| \leq \epsilon$. Then
\begin{equation}\label{equ:epsilon}
\|\varphi\| \leq \|\varphi - \varphi_\epsilon\| + \|\varphi_\epsilon\| \leq \epsilon + \|\varphi_\epsilon\| .
\end{equation}
Moreover
\begin{align*}
\|\varphi_\epsilon\| & = \left(\int_0^1 |S_\psi[\varphi_\epsilon](\alpha)|^2 [\psi, \psi](\alpha) d\alpha\right)^\frac12\\
& \leq \left(\int_0^1 |S_\psi[\varphi_\epsilon - \varphi](\alpha)|^2 [\psi, \psi](\alpha) d\alpha\right)^\frac12 + \left(\int_0^1 |S_\psi[\varphi](\alpha)|^2 [\psi, \psi](\alpha) d\alpha\right)^\frac12\\
& \leq \|\varphi - \varphi_\epsilon\| + \left(\int_0^1 |S_\psi[\varphi](\alpha)|^2 [\psi, \psi](\alpha) d\alpha\right)^\frac12
\end{align*}
where the last inequality is due to (\ref{equ:inequality}). By (\ref{equ:epsilon}) we can then deduce that
$$
\left(\int_0^1 |S_\psi[\varphi](\alpha)|^2 [\psi, \psi](\alpha) d\alpha\right)^\frac12 \geq \|\varphi_\epsilon\| - \|\varphi - \varphi_\epsilon\| \geq \|\varphi_\epsilon\| - \epsilon \geq \|\varphi\| - 2\epsilon\ .
$$
This combined with (\ref{equ:inequality}) provides then the isometry for the space $\psis{\psi}{\Gamma_2,\lr}$.

To show that $S_\psi$ is onto we argue by contradiction. Suppose that $S_\psi(\psis{\psi}{\Gamma_2,\lr}) \subsetneq L^2((0,1], [\psi,\psi](\alpha)d\alpha)$ holds. Then we can choose a nontrivial $f \in L^2((0,1], [\psi,\psi](\alpha)d\alpha)$ such that $f \bot \, S_\psi( \psis{\psi}{\Gamma_2,\lr})$, i.e.
$$
0 = \int_0^1 f(\alpha) \overline{S_\psi[\lr_{\gamma_2}\psi](\alpha)} [\psi,\psi](\alpha)d\alpha \quad \forall \ \gamma_2 \in \Gamma_2 .
$$
But since $S_\psi[\lr_{\gamma_2}\psi](\alpha) = e^{2\pi i \alpha \gamma_2} S_\psi[\psi](\alpha) = e^{2\pi i \alpha \gamma_2} \chi_{E_\psi}$ we are saying that
$$
0 = \int_0^1 e^{-2\pi i \alpha \gamma_2} f(\alpha) [\psi,\psi](\alpha)d\alpha \quad \forall \ \gamma_2 \in \Gamma_2
$$
so, by uniqueness of Fourier series coefficients and the nonnegativity of $[\psi,\psi](\alpha)$, this implies that $f = 0$ in $L^2((0,1], [\psi,\psi](\alpha)d\alpha)$, hence the contradiction.
\end{proof}

\section{Riesz bases and frames}\label{sec:Riesz bases and frames}

In this section we prove Theorems \ref{theo:Riesz} and \ref{theo:frames}. Recall that $\{\lr_{\gamma}\psi\}_{\gamma \in \Gamma}$ is a Riesz basis
for $\langle \psi \rangle_{\Gamma,L}$  with constants $A$ and $B$ if it is a basis and if for all sequences $\{a_\gamma\}_{\Gamma} \in \ell_2 (\Gamma)$  it satisfies
\begin{equation} \label{4a}
A \sum_{\gamma \in \Gamma}|a_\gamma|^2 \leq \|\displaystyle{\sum_{\gamma \in \Gamma}} a_{\gamma} \lr_{\gamma}\psi\|^2_{L^2(\HD)} \leq B \sum_{\gamma \in \Gamma}|a_\gamma|^2\,,
\end{equation}
while $\{\lr_{\gamma}\psi\}_{\gamma \in \Gamma}$ is a frame for $\psis{\psi}{\Gamma,\lr}$ with constants $A$ and $B$ if for every $\varphi \in \psis{\psi}{\Gamma,\lr}$ it holds
\begin{equation} \label{4b}
A \|\varphi\|_{L^2(\HD)} \leq \sum_{\gamma \in \Gamma} |\langle \varphi, \lr_{\gamma} \psi\rangle_{L^2(\HD)}|^2 \leq B \|\varphi\|_{L^2(\HD)} .
\end{equation}

Both theorems require condition (\ref{equ:Condition}) or its equivalent forms described by Proposition \ref{prop:Condition}. One relevant consequence of such condition is the following.

\begin{lemma}\label{orthogonality}
Let $\psi \in  L^2(\HD)$ and assume condition (\ref{equ:Condition}). Then,
\begin{itemize}
\item[i)] If $\gamma_1, \eta_1 \in \Gamma_1$ and $ \gamma_1 \neq \eta_1$, $\langle L_{\gamma_1} \psi\rangle_{\Gamma_2, L} \perp 
\langle L_{\eta_1} \psi\rangle_{\Gamma_2, L}$ and 
\begin{equation}\label{equ:directsum}
\psis{\psi}{\Gamma,\lr} = \bigoplus_{\gamma_1 \in \Gamma_1} \psis{L_{\gamma_1}\psi}{\Gamma_2,\lr} .
\end{equation}
as an orthogonal direct sum.
\item[ii)] For each $\varphi \in \psis{\psi}{\Gamma,\lr}$ there exist unique $\varphi_{\gamma_1} \in \psis{L_{\gamma_1}\psi}{\Gamma_2,\lr}\,,
\gamma_1 \in \Gamma_1\,,$ such that $\varphi = \sum_{\gamma_1 \in \Gamma_1} \varphi_{\gamma_1}$ (convergence in $L^2(\HD)$) and 
\begin{equation} \label{4c}
\| \varphi \|^2_{L^2(\HD)} = \sum_{\gamma_1 \in \Gamma_1} \| \varphi_{\gamma_1}\|^2_{L^2(\HD)}\,.
\end{equation}
\end{itemize}
\end{lemma}

\begin{proof}
Follows from Proposition \ref{prop:Condition} by easy density arguments.
\end{proof}

\begin{proposition}\label{prop:consequences}
Let $\psi$ be in $L^2(\HD)$ and assume condition (\ref{equ:Condition}). Then
\begin{itemize}
\item[i)] $\{\lr_{\gamma}\psi\}_{\gamma\in \Gamma}$ is a Riesz family with constants $A$ and $B$ if and only if $\{\lr_{\gamma_2}\psi\}_{\gamma_2 \in \Gamma_2}$ is a Riesz family with constants $A$ and $B$.
\item[ii)] $\{\lr_{\gamma}\psi\}_{\gamma \in \Gamma}$ is a frame for $\psis{\psi}{\Gamma,\lr}$ with constants $A$ and $B$ if and only if $\{\lr_{\gamma_2}\psi\}_{\gamma_2 \in \Gamma_2}$ is a frame for $\psis{\psi}{\Gamma_2,\lr}$ with constants $A$ and $B$.
\end{itemize}
\end{proposition}
\newpage
\begin{proof}
$i)$ Suppose that $\{\lr_{\gamma_2}\psi\}_{\gamma_2 \in \Gamma_2}$ is a Riesz family with constants $A$ and $B\,.$
Let $\{a_\gamma\}_{\gamma \in \Gamma} \in \ell_2 (\Gamma)\,.$ Write
$$ \varphi \doteq \sum_{\gamma \in \Gamma} a_\gamma L_\gamma \psi = \sum_{\gamma_1} \sum_{\gamma_2}  a_{\gamma_1 \gamma_2}
L_{\gamma_2} L_{\gamma_1} \psi \doteq \sum_{\gamma_1} \varphi_{\gamma_1}\,,
$$
with $\varphi_{\gamma_1} \doteq \sum_{\gamma_2} a_{\gamma_1 \gamma_2} L_{\gamma_2} L_{\gamma_1}\psi \in \psis{L_{\gamma_1}\psi}{\Gamma_2,\lr} $
since   $\{\lr_{\gamma_2}\psi\}_{\gamma_2 \in \Gamma_2}$ is a Riesz family  
and $\{a_{\gamma_1 \gamma_2}\} _{\gamma_2 \in \Gamma_2} 
\in \ell_2 (\Gamma_2)$ for each $\gamma_1 \in \Gamma_1\,.$ By (\ref{4c})
\begin{eqnarray*}
\| \sum_{\gamma \in \Gamma} a_\gamma L_\gamma \psi \|^2_{L^2(\HD)} & = & \sum_{\gamma_1}  
\| \sum_{\gamma_2} a_{\gamma_1 \gamma_2} L_{\gamma_1} L_{\gamma_2} \psi \|^2_{L^2(\HD)} \\
& = & \sum_{\gamma_1} \| \sum_{\gamma_2} a_{\gamma_1 \gamma_2}  L_{\gamma_2} \psi \|^2_{L^2(\HD)} 
\end{eqnarray*}
since $L_{\gamma_1}$ is an isometry on $L^2(\HD)\,.$ Our hypothesis implies
$$ \| \sum_{\gamma \in \Gamma} a_\gamma L_\gamma \psi \|^2_{L^2(\HD)} \thickapprox 
\sum_{\gamma_1} \sum_{\gamma_2}  |a_{\gamma_1 \gamma_2}|^2 = \sum_{\gamma \in \Gamma} |a_\gamma|^2\,,
$$
where the equivalence is up to the Riesz family constants $A$ and $B$. This shows \eqref{4a} for the family $\{\lr_{\gamma}\psi\}_{\gamma\in \Gamma}\,.$
The converse is straight forward.

$ii)$ Suppose that $\{\lr_{\gamma_2}\psi\}_{\gamma_2 \in \Gamma_2}$ is a frame for $\psis{\psi}{\Gamma_2,\lr}$ 
with constants $A$ and $B\,.$ By Lemma \ref{orthogonality} ii), choose $\varphi_{\gamma_1} \in \psis{L_{\gamma_1} \psi}{\Gamma_2,\lr},
\gamma_1 \in \Gamma_1,$ such that $\varphi = \sum_{\gamma_1} \varphi_{\gamma_1}$ (convergence in $L^2(\HD)$) and \eqref{4c} holds.
Then, by the orthogonality relations in Lemma \ref{orthogonality} i),
\begin{eqnarray*}
 \sum_{\gamma \in \Gamma} |\langle \varphi , L_\gamma \psi\rangle_{L^2(\HD)} |^2 & = & \sum_{\eta_1}
 \sum_{\eta_2} |\langle \sum_{\gamma_1} \varphi_{\gamma_1} , L_{\eta_1} L_{\eta_2}\psi \rangle_{L^2(\HD)}|^2 \\
& = & \sum_{\gamma_1}  \sum_{\eta_2} |\langle  \varphi_{\gamma_1} , L_{\gamma_1} L_{\eta_2}\psi \rangle_{L^2(\HD)}|^2 \\
& = & \sum_{\gamma_1}  \sum_{\eta_2} |\langle  L_{\gamma_1^{-1}} \varphi_{\gamma_1} ,  L_{\eta_2}\psi \rangle_{L^2(\HD)}|^2 \,.
\end{eqnarray*}
Since  $  L_{\gamma_1^{-1}} \varphi_{\gamma_1} \in \psis{ \psi}{\Gamma_2,\lr}$  our hypothesis implies
$$  \sum_{\gamma \in \Gamma} |\langle \varphi , L_\gamma \psi \rangle_{L^2(\HD)}|^2 \thickapprox
\sum_{\gamma_1} \|L_{\gamma_1^{-1}} \varphi_{\gamma_1}\|^2_{L^2(\HD)} = \sum_{\gamma_1} \|\varphi_{\gamma_1}\|^2_{L^2(\HD)} =
\| \varphi \|^2_{L^2(\HD)},
$$
by \eqref{4c}where the equivalence is up to the frame constants $A$ and $B$. This proves that $\{\lr_{\gamma}\psi\}_{\gamma \in \Gamma}$ is a frame for $\psis{\psi}{\Gamma,\lr}$ with constants $A$ and $B.$

For the converse, let $\varphi \in \psis{\psi}{\Gamma_2,\lr}.$ By the orthogonality relation in  Lemma \ref{orthogonality} i), $\langle \varphi ,
L_{\gamma_1} L_{\gamma_2}\psi \rangle=0$ for all $\gamma_1 \in \Gamma_1, \gamma_1 \neq 0.$ Then
$$ \sum_{\gamma_2} |\langle \varphi , L_{\gamma_2} \psi \rangle_{L^2(\HD)}|^2 =
\sum_{\gamma_1} \sum_{\gamma_2} |\langle \varphi , L_{\gamma_1} L_{\gamma_2} \psi \rangle_{L^2(\HD)}|^2 \thickapprox 
\| \varphi \|^2_{L^2(\HD)},
$$
since $ \varphi \in \psis{\psi}{\Gamma_2,\lr} \subset \psis{\psi}{\Gamma,\lr},$ where the equivalence is up to the frame constants $A$ and $B.$
\end{proof}

We are now ready to prove the two main theorems.
\begin{proof}[Proof of Theorem \ref{theo:Riesz}]
Let us first assume $b)$. Observe that by Lemma \ref{lem:unitaryhom}, $i)$, $b)$ is equivalent to
$$
A \leq [\lr_{\gamma_1}\psi, \lr_{\gamma_1}\psi](\alpha) \leq B \quad \textnormal{a.e.} \ \alpha \in (0,1] \, , \forall \ \gamma_1 \in \Gamma_1\,.
$$
This implies that $E_{\psi_{\gamma_1}} = E_\psi = (0,1] $ a.e. Hence, by Lemma (\ref{finite-sum}) $i)$
$$
S_{\lr_{\gamma_1}\psi} [\lr_{\gamma_2}\lr_{\gamma_1}\psi](\alpha) = e^{2\pi i \alpha \gamma_2} \quad \textnormal{a.e.} \ \alpha \in (0,1] .
$$
Moreover, we have that $L^2((0,1],[\lr_{\gamma_1}\psi, L_{\gamma_1}\psi](\alpha) d\alpha) \approx L^2((0,1], d\alpha)$ for all $\gamma_1 \in \Gamma_1$, where $\approx$ stands for equivalence up to constants $A$ and $B$. Since, by Theorem \ref{theo:Isometry}, $S_{L_{\gamma_1}\psi}$ is a linear isometry from $\psis{L_{\gamma_1}\psi}{\Gamma_2,\lr}$ onto $L^2((0,1],[L_{\gamma_1}\psi, L_{\gamma_1}\psi](\alpha) d\alpha)$, and since $\{e^{2\pi i \alpha \gamma_2}\}_{\gamma_2 \in \Gamma_2}$ is an orthonormal basis of $L^2((0,1], d\alpha)$, we conclude that $\{\lr_{\gamma_2}\lr_{\gamma_1}\psi\}_{\gamma_2 \in \Gamma_2}$ is a basis of $\psis{\lr_{\gamma_1}\psi}{\Gamma_2,\lr}$ for all $\gamma_1 \in \Gamma_1$. Thus by (\ref{equ:directsum}) we have that $\{\lr_{\gamma}\psi\}_{\gamma \in \Gamma}$ is a basis of $\psis{\psi}{\Gamma,\lr}$.

Now, using Proposition \ref{prop:consequences} i), to conclude the proof that $b)$ implies $a)$ and to prove that $a)$ implies $b)$ we need only to prove that the following conditions are equivalent:
\begin{itemize}
\item[a')] $\{\lr_{\gamma_2}\psi\}_{\gamma_2 \in \Gamma_2}$ is a Riesz family for $\psis{\psi}{\Gamma_2,\lr}$ with constants $A$ and $B$
\item[b)] $A \leq [\psi,\psi](\alpha) \leq B$ for a.e. $\alpha \in (0,1]$.
\end{itemize}

If $b)$ holds, and $\{c_{\gamma_2}\}_{\gamma_2 \in \Gamma_2} \in \ell_2 (\Gamma_2),$  by Theorem \ref{theo:Isometry} and 
Lemma \ref{finite-sum},
\begin{eqnarray*}
\|\sum_{\gamma_2} c_{\gamma_2} \lr_{\gamma_2} \psi\|^2_{L^2(\HD)} & = & \bigg\|S_\psi\bigg[\sum_{\gamma_2} c_{\gamma_2} \lr_{\gamma_2} \psi\bigg]\bigg\|_{L^2((0,1],[\psi,\psi](\alpha)d\alpha)}\\ &\approx & \bigg\|S_\psi\bigg[\sum_{\gamma_2} c_{\gamma_2} \lr_{\gamma_2} \psi\bigg]\bigg\|_{L^2((0,1],d\alpha)} \\
& = & \int_0^1 |\sum_{\gamma_2} c_{\gamma_2} e^{2\pi i \gamma_2 \alpha}|^2 d \alpha = \sum_{\gamma_2} |c_{\gamma_2}|^2\,,
\end{eqnarray*}
where equivalence is up to constants $A$ and $B$. Thus we have obtained $a')$.

Conversely, let us suppose that $a')$ holds and proceed by contradiction. Let us assume that $[\psi,\psi](\alpha) < A$ on a set $E \subset (0,1]$ of positive measure. Since $\chi_E \in L^2((0,1],d\alpha)$, there exists a sequence $\{c_{\gamma_2}\}_{\gamma_2 \in \Gamma_2} \in  \ell_2 (\Gamma_2)$ such that
$$
\chi_E(\alpha) = \displaystyle{\sum_{\gamma_2}c_{\gamma_2}e^{2\pi i \gamma_2 \alpha}} \, , \quad \textnormal{and} \ |E| = \int_0^1 |\chi_E(\alpha)|^2 d\alpha = \sum_{\gamma_2} |c_{\gamma_2}|^2 .
$$
But, again using Theorem \ref{theo:Isometry} and Lemma \ref{finite-sum}, $ii)$
\begin{eqnarray*}
\|\sum_{\gamma_2} c_{\gamma_2} \lr_{\gamma_2} \psi\|^2_{L^2(\HD)} & = & \int_0^1 \bigg|S_\psi\bigg[\sum_{\gamma_2} c_{\gamma_2} \lr_{\gamma_2} \psi\bigg](\alpha)\bigg|^2 [\psi,\psi](\alpha) d\alpha\\
& = & \int_0^1 |\chi_E(\alpha)|^2 [\psi,\psi](\alpha) d\alpha < A |E|= A \sum_{\gamma_2} |c_{\gamma_2}|^2
\end{eqnarray*}
which contradicts $a')$. This proves that  $[\psi,\psi](\alpha) \geq A$ for a.e. $\alpha \in (0,1].$ 
The proof that $[\psi,\psi](\alpha) \leq B$ for a.e. $\alpha \in (0,1]$ is similar.
\end{proof}

\begin{proof}[Proof of Theorem \ref{theo:frames}]
As in the previous proof, by Proposition \ref{prop:consequences} it suffices to prove that $b)$ is equivalent to

$a')$ $\{\lr_{\gamma_2}\psi\}_{\gamma_2 \in \Gamma_2}$ is a frame for $\psis{\psi}{\Gamma_2,\lr}$ with constants $A$ and $B$.

Let us first assume $b)$, and choose a $\varphi \in \psis{\psi}{\Gamma_2,\lr}$. Then by Theorem \ref{theo:Isometry} and Lemma \ref{finite-sum}, $i)$
\begin{eqnarray}
\sum_{\gamma_2 \in \Gamma_2} |\langle \varphi, \lr_{\gamma_2}\psi\rangle_{L^2(\HD)}|^2 & = & \sum_{\gamma_2 \in \Gamma_2} |\langle S_\psi[\varphi], S_\psi[\lr_{\gamma_2}\psi]\rangle_{L^2((0,1],[\psi,\psi](\alpha)d\alpha}|^2\nonumber\\
& = & \sum_{\gamma_2 \in \Gamma_2} | \int_0^1 S_\psi[\varphi](\alpha) e^{-2\pi i \alpha \gamma_2} [\psi,\psi](\alpha) d\alpha |^2\nonumber\\
& = & \int_0^1 |S_\psi[\varphi](\alpha)|^2 \, |[\psi,\psi](\alpha)|^2 d\alpha\label{equ:dummy}
\end{eqnarray}
and the last transition is due to Parseval's identity. So if $A \leq [\psi,\psi](\alpha) \leq B$ for a.e. $\alpha \in E_\psi$, we get
$$
\sum_{\gamma_2 \in \Gamma_2} |\langle \varphi, \lr_{\gamma_2}\psi\rangle_{L^2(\HD)}|^2 \approx \int_0^1 |S_\psi[\varphi](\alpha)|^2 \, [\psi,\psi](\alpha) d\alpha = \|\varphi\|_{L^2(\HD)}
$$
where the last identity is due to Theorem \ref{theo:Isometry} and equivalence is up to constants $A$ and $B$.

Now let us assume $a')$ and suppose by contradiction that $[\psi,\psi](\alpha) < A$ on a set $E \subset E_\psi$ of positive measure. Observe that $\chi_E \in L^2((0,1], [\psi,\psi](\alpha)d\alpha)$ because $[\psi,\psi] \in L^1((0,1],d\alpha)$. Since $S_\psi$ is onto, we can then choose $\varphi_E \in \psis{\psi}{\Gamma_2,\lr}$ such that $S_\psi[\varphi_E] = \chi_E$. As for (\ref{equ:dummy})
\begin{eqnarray*}
\sum_{\gamma_2 \in \Gamma_2} |\langle \varphi_E, \lr_{\gamma_2}\psi\rangle_{L^2(\HD)}|^2 & = & \int_0^1 |S_\psi[\varphi_E](\alpha)|^2 \, |[\psi,\psi](\alpha)|^2 d\alpha \\
& = & \int_0^1 |\chi_E(\alpha)|^2 \, |[\psi,\psi](\alpha)|^2 d\alpha < A \int_E [\psi,\psi](\alpha) d\alpha
\end{eqnarray*}
but since $S_\psi$ is an isometry
$$
\int_E [\psi,\psi](\alpha) d\alpha = \int_0^1 |S_\psi[\varphi_E](\alpha)|^2 \, [\psi,\psi](\alpha) d\alpha = \|\varphi_E\|^2_{L^2(\HD)}\,,
$$
hence contradicting $a')$. This proves that  $[\psi,\psi](\alpha) \geq A$ for a.e. $\alpha \in E_\psi.$
The proof that $[\psi,\psi](\alpha) \leq B$ for a.e. $\alpha \in E_\psi$ is similar.
\end{proof}

\section{Applications}\label{sec:Applications}

In this section we show as an application of the previous theory how frames for cyclic subspaces can be associated to orthonormal Gabor systems
$\in L^2(\mathbb R^d).$

Let $u$ be a unit norm element in $L^2(\R^d)$, and call $u_\lambda$ its normalized scalings
$$
u_\lambda(x) \doteq |\lambda|^{\frac{d}{2}} u(\lambda x) \, ,\quad \lambda \in \R^*.
$$
Let $\P_\lambda$ be the associated projector on $L^2(\R^d)$
$$
\P_\lambda = u_\lambda \otimes u_\lambda \, , \quad \P_\lambda : L^2(\R^d) \ni f \mapsto \langle f, u_\lambda\rangle_{L^2(\R^d)} u_\lambda
$$
and define, for all $\epsilon \in (0,1)$
$$
H_{\epsilon}(\lambda) = \left\{
\begin{array}{cl}
\P_\lambda & \epsilon < \lambda \leq 1\\
0 & \textnormal{otherwise} \ .
\end{array}
\right.
$$
For all $\epsilon \in (0,1)$, $H_{\epsilon}$ belongs to $L^2(\R^*,\HS(L^2(\R^d)),|\lambda|^d d\lambda)$, since
$$
\|H_{\epsilon}(\lambda)\|_{\HS(L^2(\R^d))} = \left\{
\begin{array}{cl}
1 & \epsilon < \lambda \leq 1\\
0 & \textnormal{otherwise}
\end{array}
\right.
$$
so
$$
\|H_{\epsilon}(\lambda)\|^2_{L^2(\R^*,\HS(L^2(\R^d)),|\lambda|^d d\lambda)} = \int_{\R} \|H_{\epsilon}(\lambda)\|^2_{\HS(L^2(\R^d))}|\lambda|^d d\lambda = \int_\epsilon^1 \lambda^d d\lambda = \frac{1 - \epsilon^{d + 1}}{d + 1} \, .
$$

We can then define $\psi_\epsilon \in L^2(\HD)$ as $\psi_\epsilon = \F_\HD^{-1} H_\epsilon$, and prove the following theorem.

\begin{theorem}\label{theo:example}
Let $\{\uir{\alpha}(\gamma_1)u_\alpha\}_{\gamma_1 \in \Gamma_1}$ be an orthonormal system in $L^2(\R^d)$ for a.e. $0 < \epsilon < \alpha \leq 1$. Then $\{\lr_\gamma \psi_\epsilon\}$ is a frame for $\psis{\psi_\epsilon}{\Gamma,\lr}$ for all $\epsilon \in (0,1)$.
\end{theorem}

In order to prove Theorem \ref{theo:example} we start with two preliminary lemmata.

\begin{lemma}\label{lem:example1}
For all $\epsilon \in (0,1)$ and all $\alpha \in E_{\psi_\epsilon}$
$$
\epsilon^d \leq [\psi_\epsilon, \psi_\epsilon](\alpha) \leq 1 .
$$
\end{lemma}
\begin{proof}
Since $\psi_\epsilon$ is bandlimited
\begin{eqnarray*}
[\psi_\epsilon, \psi_\epsilon](\alpha) & = & \sum_{j \in \Z} \| \F_\HD \psi_\epsilon (\alpha + j)\|^2_{\HS(L^2(\R^d))}|\alpha + j|^d\\
& = & \sum_{j \in \Z} \|H_{\epsilon}(\alpha + j)\|^2_{\HS(L^2(\R^d))}|\alpha + j|^d\\
& = & \|H_{\epsilon}(\alpha)\|_{\HS(L^2(\R^d))}|\alpha|^d \ = \
\left\{
\begin{array}{cl}
\alpha^d & \epsilon < \alpha \leq 1\\
0 & \textnormal{otherwise}\ .
\end{array}
\right.
\end{eqnarray*}
\end{proof}

\begin{lemma}\label{lem:example2}
For all $\epsilon \in (0,1)$, all $\gamma_1 \in \Gamma_1$ and all $\alpha \in E_{\psi_\epsilon}$
$$
[\lr_{\gamma_1}\psi_\epsilon, \psi_\epsilon](\alpha) = \langle \uir{\alpha}(\gamma_1) u_{\alpha}, u_{\alpha}\rangle_{L^2(\R^d)} .
$$
\end{lemma}
\begin{proof}
Again using the bandlimitness of $\psi_\epsilon$ to reduce the series to the only nonzero term
\begin{eqnarray*}
[\lr_{\gamma_1}\psi_\epsilon, \psi_\epsilon](\alpha) & = & \sum_{j \in \Z} \langle \F_\HD \lr_{\gamma_1}\psi_\epsilon (\alpha + j), \F_\HD \psi_\epsilon (\alpha + j)\rangle_{\HS(L^2(\R^d))}|\alpha + j|^d\\
& = & \sum_{j \in \Z} \langle \uir{\alpha + j}(\gamma_1) H_{\epsilon}(\alpha + j), H_{\epsilon}(\alpha + j)\rangle_{\HS(L^2(\R^d))}|\alpha + j|^d\\
& = & \langle \uir{\alpha}(\gamma_1) H_{\epsilon}(\alpha), H_{\epsilon}(\alpha)\rangle_{\HS(L^2(\R^d))}|\alpha|^d .
\end{eqnarray*}
Now, since $\uir{\alpha}(\gamma_1) \P_\alpha = (\pi_{\alpha}(\gamma_1) u_\alpha)\otimes u_\alpha$, when $\epsilon < \alpha \leq 1$ we have
\begin{eqnarray*}
\langle \pi_{\alpha}(\gamma_1) H_{\epsilon}(\alpha), H_{\epsilon}(\alpha)\rangle_{\HS(L^2(\R^d))} & = & \int_{\R^d} \int_{\R^d} \overline{u_\alpha(x)} (\pi_{\alpha}(\gamma_1)u_\alpha(y)) \overline{u_\alpha(y)} u_\alpha(x) dx dy\\
& = & \langle \pi_{\alpha}(\gamma_1) u_{\alpha}, u_{\alpha}\rangle_{L^2(\R^d)} .
\end{eqnarray*}
\end{proof}

\begin{proof}[Proof of Theorem \ref{theo:example}]
In order to get a frame, we need only to verify condition $b)$ of Theorem \ref{theo:frames}, which is granted by Lemma \ref{lem:example1}, and condition (\ref{equ:Condition}), that holds by Lemma \ref{lem:example2} whenever $\{\pi_\alpha(\gamma_1)u_\alpha\}_{\gamma_1 \in \Gamma_1}$ is an orthonormal system.
\end{proof}


\begin{example} 
One example of unitary $u \in L^2(\mathbb R^d)$ and $\Gamma_1$ for which $\{\pi_\alpha (\gamma_1) u_\alpha\}_{\gamma_1\in \Gamma_1}$
is an orthonormal system of $L^2(\mathbb R^d)$ for $0 < \epsilon < \alpha \leq 1$ is the following. Choose $u = \chi_{(0,1)}\times \cdots\times \chi_{(0,1)}\in L^2(\R^d)$  and 
$\Gamma_1 = a\mathbb Z^d \times b\mathbb Z^d$ with $a,b \in \mathbb Z$. 
 The desired orthonormality follows from
$$ \pi_\alpha (\gamma_1) u_\alpha(y) = |\alpha|^{d/2}e^{-2\pi i\alpha b\langle m, y\rangle} \chi_{(0,\frac{1}{\alpha})\times \cdots \times (0,\frac{1}{\alpha})} (y-an)$$
for $\gamma_1 = (an, bm,0) \in \Gamma_1\,.$

\end{example}

\section{$SI/Z$ groups}\label{sec:SIZ}
In this section we consider a class of nilpotent groups distinct from the stratified groups: following \cite{corwin}, we say that  a simply connected nilpotent Lie group $\G$ is an {\it $SI/Z$ group} if almost all of its irreducible representations  $\Pi$ are square-integrable modulo the center of the group. An irreducible representation $\Pi$ of $\G$ on the Hilbert space $\H$ is {\it square-integrable modulo the center} if there exist $u, v \in \H$ such that
$$
\int_{\G/Z}\left\vert \left\langle \Pi\left(  n\right)  u,v\right\rangle_\H
\right\vert ^{2}d\dot n  <\infty.
$$
 
We remark that the class of $SI/Z$ groups is broad and in particular contains the Heisenberg group as well as groups of an arbitrarily high degree of non-commutativity.

The effect of the $SI/Z$  condition is that for this class of groups, the operator-valued Plancherel transform retains certain key features of the Euclidean case, and in particular, makes it possible to extend to $SI/Z$ groups the 
notion of bracket map. The following construction, leading to the Plancherel theorem, can be found in \cite{CMO, corwin}.

For a $SI/Z$ group $\G$, fix a basis $\{ X_1, \dots, X_n\}$ for its Lie algebra $\g$, ordered in such a way that $\g_j = \textnormal{span}\{X_1, \dots, X_j\}$ is an ideal in $\g$ (central series) and $\g_r$ is the center $\z$ of $\g$ for some $r= n-2d$. Then the center $Z$ of $\G$ can be identified with $\R^r$ by choosing exponential coordinates
$$
Z = e^{z_1X_1} e^{z_2X_2}\cdots e^{z_rX_r}
$$
where $(z_1, z_2, \dots, z_r)\in \R^r$, while the subset of $\G$
$$
{\mathcal X} = e^{\R X_{r+1}}e^{\R X_{r+2}}\cdots e^{\R X_{n}}
$$
is identified by $\R^{2d}$. For any $g\in \G$ there is a unique $x\in {\mathcal X}$ and a unique $z\in Z$ such that $g = xz = zx$, therefore the group $\G$ is identified by $\R^{2d} \times \R^r$.

\newpage

Let $\g^*$ denote the linear dual of $\g$, and let $\z^*$ denote the subspace
$$
\z^* = \{\lambda \in \g^* \ | \ \lambda(X_j) = 0 \, , r < j \leq n\} .
$$
Clearly $\z^* \approx \z \approx \R^r \approx Z$ as topological spaces. Let us also introduce the so-called Pfaffian determinant
$$
\begin{array}{rccl}
\rho : & \z^* & \rightarrow & \R^+\\
& \lambda & \mapsto & \left| \det \Big( \lambda([X_i, X_j])\Big)_{i,j = r, \dots, n}\right|
\end{array}
$$
and denote with $\g_j(\lambda)$ the maximal subalgebra $m_j$ of $\g_j$ such that $\lambda([m_j,m_j]) = 0$.
Here $[\cdot,\cdot]$ stands for the Lie algebra commutator. With this notation, the following Plancherel theorem holds.
\begin{theorem}
Let $\G$ be an $SI/Z$ group and let $\Sigma = \{\lambda \in \z^* \ | \ \rho(\lambda) > 0\}$. Moreover, let $\Pi_\lambda$ be the representation of $\G$ induced by the subalgebra
$$
P(\lambda) = \sum_{j = 1}^n \g_j(\lambda) \, ,\quad \lambda \in \Sigma .
$$
Then $\Pi_\lambda$ is a unitary irreducible representation on $L^2(\R^d)$, and the group Fourier transform
$$
\F_\G \varphi (\lambda) = \int_\G \varphi(g) \Pi_\lambda(g) dg
$$
where $dg$ is the Haar measure of $\G$, implements an isometric isomorphism - the Plancherel transform - 
$$
\F_\G : L^2(\G) \rightarrow L^2(\z^*, \HS(L^2(\R^d)), \rho(\lambda) d\lambda)
$$
where $d\lambda$ is the Lebesgue measure on $\z^* \approx \R^r$.
\end{theorem}

This theorem actually provides all we need to define a bracket like (\ref{equ:Hbracket}). More precisely we can extend such definition to
\begin{equation}\label{equ:Gbracket}
[\varphi, \psi](\alpha) = \sum_{j \in \Z^r} \langle \F_\G \varphi(\alpha + j), \F_\G \psi(\alpha + j)\rangle_{\HS} \, \rho(\alpha + j)
\end{equation}
for $\varphi, \psi \in L^2(\G)$ and a.e. $\alpha \in \T^r$, the $r$ dimensional torus, that can be identified with $[0, 1)^r$ in $\R^r$ as usual.

Indeed, the bracket (\ref{equ:Gbracket}) is related to scalar product of $L^2(\G)$ in the same way as (\ref{equ:Hbracket}) was, i.e. by means of Plancherel theorem and a periodization over $\R^r$:
\begin{align*}
\langle \varphi, \psi\rangle_{L^2(\G)} & = \int_{\R^r} \langle \F_G\varphi(\lambda), \F_G\psi(\lambda)\rangle_{\HS(L^2(\R^d))} \rho(\lambda) d\lambda\\
& = \sum_{j \in \Z^r} \int_{\T^r} \langle \F_G\varphi(\alpha + j), \F_G\psi(\alpha + j)\rangle_{\HS(L^2(\R^d))} \rho(\alpha + j) d\alpha .
\end{align*}
The left translation  $L$ is a unitary representation of $\G$ onto $L^2(\G)$ and is defined by the group action. 
The Plancherel transform intertwines the left translation operator $L_{xz}, \ xz \in \G$ with the representation $\pi_\lambda$: for $\phi \in L^1(\G)\cap L^2(\G)$

$$
\mathcal F_\G{\bigl(L(xz) \phi\bigr)}(\lambda) = \pi_\lambda(xz) \bigl(\mathcal F_\G\phi\bigr)(\lambda) = e^{2\pi i \langle \lambda,  z\rangle} \pi_\l(x) \mathcal F_\G\phi(\lambda)\ |\rho(\l)|^{1/2}, \  \ \l \in \z^*.
$$



\newpage

By the same arguments of the previous sections, then, one can show that the results of this paper are also true for $SI/Z$ groups when we consider a lattice subgroup $\Gamma$ of $\G$ given by following form
$$
\Gamma = \{ \gamma_1\gamma_2\in \G: \gamma_1 \in \Gamma_1, \gamma_2 \in \Gamma_2\}
$$
where $\Gamma_2$ is the integer lattice $\Gamma_2 = \exp \Z X_1 \exp \Z X_2 \cdots \exp \Z X_r$ and $\Gamma_1$ is a compatible discrete subset of $\mathcal X$. To avoid repeating the calculations, we leave the proof of the results for the $SI/Z$ groups to the interested readers.

\

\

\textsc{Davide Barbieri}:
CAMS, EHESS, Paris, France, \emph{davide.barbieri8@gmail.com}

\textsc{Eugenio Hern\'andez}:
Departamento de Matem\'aticas, Universidad Aut\'onoma de Madrid, 28049, Madrid, Spain, \emph{eugenio.hernandez@uam.es}

\textsc{Azita Mayeli}:
City University of New York,  \emph{AMayeli@qcc.cuny.edu}

\end{document}